\documentclass[11pt]{amsart}
\usepackage[utf8]{inputenc}
\usepackage{amsmath,amssymb,amsthm}
\usepackage[dvipsnames]{xcolor}
\usepackage{graphicx}
\usepackage[font=small,labelfont=sc,labelsep=period,justification=justified]{caption}
\usepackage{array,booktabs}
\usepackage[margin=2.6cm]{geometry}
\usepackage{enumitem}
\usepackage[hidelinks]{hyperref}

\newtheorem{theorem}{Theorem}[section]
\newtheorem{proposition}[theorem]{Proposition}
\newtheorem{lemma}[theorem]{Lemma}

\theoremstyle{definition}
\newtheorem{definition}[theorem]{Definition}
\newtheorem{example}[theorem]{Example}
\theoremstyle{remark}
\newtheorem{remark}[theorem]{Remark}

\newcommand{\lam}[2]{\lambda^{#1}_{#2}}
\newcommand{\R}{\mathbb{R}}
\newcommand{\T}{\mathcal{T}}
\newcommand{\Sp}{\mathcal{S}}
\newcommand{\Ex}{\mathcal{E}}

\begin{document}

\vspace{1em}

\title[Continuous envy-free allocation via reciprocal preference matrices]{Continuous envy-free allocation via reciprocal preference matrices: an alternative to Sperner-based discretization}

\author{Luis \'Angel Calvo}
\address{Luis \'Angel Calvo Pascual, Department of Quantitative Methods, ICADE and Institute for Research in Technology (IIT), ICAI,
Comillas Pontifical University, Spain.}
\email{lacalvo@comillas.edu}
\author{Susana Merch\'an}
\address{Susana Merch\'an Rubira, Department of Applied Mathematics and Computer Science for Civil and Naval Engineering, Universidad Polit\'ecnica de Madrid, Spain.}
\email{susana.merchan@upm.es}
\author{Diego de Pereda}
\address{Diego de Pereda Sebasti\'an, Department of Mathematics and Physics, Universidad Nebrija, Spain.}
\email{dpereda@nebrija.es}
\author{Javier Rodrigo}
\address{Javier Rodrigo Hitos, Department of Applied Mathematics, ICAI, Comillas Pontifical University, Spain.}
\email{jrodrigo@comillas.edu}

\begin{abstract}
Fair allocation problems on a continuous simplex are often approached through Sperner-type discretizations: agents label mesh vertices according to their preferences and a fully labelled elementary triangle gives a discrete localization of a fair outcome. We develop a continuous alternative for three agents and three alternatives by representing preferences with multiplicatively consistent reciprocal pairwise-comparison matrices. Two independent ratios per agent induce exact linear indifference boundaries and turn envy-free allocation into a finite-dimensional feasibility problem. For the six possible assignments, we introduce an assignment score $\Phi$ built from the agents' consistent priority vectors and prove that an incentive solution component is nonempty if and only if its assignment is the unique minimizer of $\Phi$, whereas a disincentive solution component is nonempty if and only if its assignment is the unique maximizer. Consequently, the strict solution region is either empty or a unique convex preference region whose closure lies in the interior of the simplex. The weak solution set is never empty and degenerates exactly when $\Phi$ has ties. We also compare the exact continuous formulation with Sperner-based discretization, including the boundary convention required in the disincentive model. Under exact consistency, the model uses six independent pairwise thresholds, whereas a regular triangular mesh with resolution $m$ contains $(m+1)(m+2)/2$ vertices. The method therefore exchanges a stronger parametric preference assumption for an exact continuous solution set and a resolution-independent elicitation stage.
\end{abstract}


\maketitle

\section{Introduction}

Fair allocation problems combine two tasks that are often separated in mathematical models: the first is to describe how agents compare feasible alternatives, and the second is to identify an allocation at which these preferences are simultaneously satisfied. In the three-agent setting studied here, feasible allocations are represented by the standard simplex
\[
T=\{(a,b,c)\in\R^3_{\ge 0}: a+b+c=1\},
\]
where the three coordinates are associated with alternatives $A$, $B$, and $C$. The same geometry covers two practically different interpretations. In the incentive model, the coordinates represent benefits, so larger values are preferred; salaries attached to three jobs are a simple example. In the disincentive model, the coordinates represent costs or burdens, so smaller values are preferred; rental harmony, where the coordinates are room prices, is the canonical example. In both cases the three alternatives are assigned to the three agents, and an allocation together with an assignment is \emph{weakly envy-free} if no agent strictly prefers the alternative assigned to another agent, and \emph{strictly envy-free} if every agent strictly prefers her own alternative. Fair division has a long mathematical tradition, going back to Steinhaus~\cite{Steinhaus}, Dubins and Spanier~\cite{DubinsSpanier} and Stromquist~\cite{Stromquist}, and modern work increasingly emphasizes not only existence but also the information and computational steps needed to implement a fair rule~\cite{AzizMackenzie,LiuLuSuzukiWalsh,Moulin,RobertsonWebb}.

Sperner's lemma \cite{Sperner} provides one of the most elegant routes from local preferences to a fair allocation. In Su's rental-harmony construction~\cite{Su}, the simplex is triangulated, that is, subdivided into finitely many small triangles, called \emph{elementary triangles}, any two of which are disjoint or meet in a common vertex or a common edge; mesh vertices are assigned to agents, and each owned vertex is labelled by a preferred alternative. A Sperner boundary condition then guarantees at least one \emph{fully labelled} elementary triangle, that is, an elementary triangle whose three vertices carry the three different labels $A$, $B$ and $C$ (the precise boundary condition is stated in Section~\ref{sec:sperner}). This gives a finite combinatorial localization of a fair outcome and connects fair division with the wider simplicial and fixed-point tradition~\cite{DeLoeraPetersonSu,Meunier,Scarf,Todd}. The practical issue considered in this paper arises when that existence argument is used as a numerical decision procedure. A mesh resolution must be selected, preferences must be elicited at its vertices, and the discrete output is a fully labelled triangle rather than the exact continuous set of allocations at which the agents' preferences are simultaneously compatible. Finer triangulations improve localization but require more queries, while a local refinement of a previously identified triangle does not automatically inherit the boundary conditions required for a new application of Sperner's lemma. These are not failures of Sperner's lemma; they are consequences of using a discrete topological certificate as a finite-resolution computational device.

We propose a different route, inspired by pairwise-comparison decision theory.

Each agent $i$ is represented by a positive reciprocal matrix whose entries encode relative preference thresholds: the entry $\lam{i}{AB}$, for instance, is the ratio $a/b$ at which agent $i$ is indifferent between $A$ and $B$. Under multiplicative consistency there exists a positive \emph{priority vector} $w^i=(w^i_A,w^i_B,w^i_C)$ such that
\[
\lam{i}{AB}=\frac{w^i_A}{w^i_B},\qquad \lam{i}{AC}=\frac{w^i_A}{w^i_C},\qquad \lam{i}{BC}=\frac{w^i_B}{w^i_C},
\]
and therefore only two of the three thresholds are independent. For an allocation $(a,b,c)\in T$ we use the normalized scores $s^i_A=a/w^i_A$, $s^i_B=b/w^i_B$, $s^i_C=c/w^i_C$. In the incentive model agent $i$ selects the largest score; in the disincentive model, the smallest. In both models the pairwise indifference boundaries are the same three lines $a=\lam{i}{AB}b$, $a=\lam{i}{AC}c$, $b=\lam{i}{BC}c$, although, as shown in Section~\ref{sec:model}, the two models use different halves of these lines to form the preference partition. Thus the reciprocal matrix is not used merely to rank alternatives: its ratios generate an exact continuous partition of the allocation simplex.

For a bijection $\pi$ from the agents to $\{A,B,C\}$, let $\Sp^{+}_{\pi}$ (respectively $\Sp^{-}_{\pi}$) be the set of allocations at which every agent strictly prefers the alternative assigned by $\pi$ in the incentive (respectively disincentive) model, and let $\Ex^{\pm}_{\pi}$ be its weak counterpart. These are exactly the strictly and weakly envy-free allocations supporting $\pi$. Each of them is the intersection of $T$ with six linear half-planes, so once the thresholds have been elicited the solution region is described exactly and no triangulation parameter is needed. With the \emph{assignment score} $\Phi(\pi)=\prod_{r=1}^3 w^r_{\pi(r)}$, our main results can be summarized as follows.

\begin{enumerate}[label=(\roman*),leftmargin=2em]
\item $\Sp^{+}_{\pi}\neq\emptyset$ if and only if $\pi$ is the unique minimizer of $\Phi$, and $\Sp^{-}_{\pi}\neq\emptyset$ if and only if $\pi$ is the unique maximizer; hence the strict solution region is either empty or a single convex region whose closure lies in the interior of $T$ (Theorem~\ref{thm:main}).
\item $\Ex^{+}_{\pi}\neq\emptyset$ if and only if $\pi$ minimizes $\Phi$ and $\Ex^{-}_{\pi}\neq\emptyset$ if and only if $\pi$ maximizes $\Phi$. The extended solution set is therefore never empty, and it degenerates to a segment or a point exactly when the optimum of $\Phi$ is attained by several assignments (Proposition~\ref{prop:weak}).
\item Fully labelled triangles of Sperner discretizations converge to the extended solution set as the mesh is refined (Proposition \ref{prop:liminf}), but at a fixed resolution they need not meet the exact solution region (Examples~\ref{ex:inc} and~\ref{ex:dis}).
\end{enumerate}

The contribution of this work lies at the intersection of fair division, decision theory, and continuous optimization. We use reciprocal matrices as a compact preference-elicitation device, reinterpret their ratios as indifference thresholds on a simplex, derive an exact existence and uniqueness theory for the resulting fair-allocation region and compare the continuous construction with a classical simplicial benchmark. The scope is deliberately restricted to three agents, three alternatives, multiplicative consistency, and linear indifference boundaries. These assumptions are what make the exact characterization possible; inconsistent matrices, curved indifference boundaries and more than three agents are natural extensions rather than features silently absorbed by the present model. In particular, linear indifference boundaries encode the assumption that each agent's valuation of an alternative is proportional to its coordinate, and the comparison with Sperner-based procedures in Section~\ref{sec:sperner} must be read in the light of this stronger assumption.

The paper is organized as follows. Section~\ref{sec:lit} reviews the related fair-division, simplicial, pairwise-comparison and rent-division literature. Section~\ref{sec:model} introduces the model and the geometric definitions. Section~\ref{sec:main} contains the exact feasibility, assignment-score, interiority and uniqueness results. Section~\ref{sec:degenerate} characterizes the degenerate configurations. Section~\ref{sec:sperner} compares the continuous construction with Sperner-based discretization, including boundary labellings, convergence, two worked examples, local refinement and elicitation counts. Section~\ref{sec:concl} states the conclusions and future lines.

\section{Related literature}\label{sec:lit}

\subsection{Fair division, Sperner labellings, and simplicial methods}
The present model combines two literatures that are usually developed separately. The first is the theory of fair division. Early formulations of divisible-resource allocation go back to Steinhaus~\cite{Steinhaus}; continuous cake-cutting procedures and envy-free existence results were developed by Dubins and Spanier~\cite{DubinsSpanier} and Stromquist~\cite{Stromquist}; and constructive and algorithmic treatments were systematized in~\cite{BramsTaylor,RobertsonWebb}. More recent work has emphasized computational complexity, query models, and implementation issues~\cite{AzizMackenzie,LiuLuSuzukiWalsh,Moulin}. This last perspective is especially relevant here because our comparison with Sperner-based procedures concerns not only existence, but also how preference information enters the computation.

The second line is the use of simplicial and labelling arguments in equilibrium and allocation problems. Su's rental-harmony construction~\cite{Su} is the direct fair-division benchmark for this paper: a triangulated simplex and a Sperner labelling convert local room preferences into a fully labelled triangle. Related simplicial methods have a much broader optimization history. Scarf's approximation of fixed points~\cite{Scarf} and Todd's lecture notes on triangulation-based fixed-point algorithms~\cite{Todd} are classical examples. Edgewise subdivision provides a systematic way to refine simplices~\cite{EdelsbrunnerGrayson}; polytopal and combinatorial extensions of Sperner's lemma appear in~\cite{DeLoeraPetersonSu,Meunier}; and recent structural work continues to investigate the critical combinatorics of Sperner labellings~\cite{KaiserStehlikSkrekovski}. Within the optimization literature, van der Laan, Talman, and Yang~\cite{vanderLaanTalmanYang} use integer labellings for discrete nonlinear equations, while Le et al.~\cite{LeLeVanPhamSaglam} derive the Gale--Nikaido--Debreu lemma directly from Sperner's lemma. Gourdel et al.~\cite{GourdelLeVanPhamTranViet} further relate Gale--Nikaido--Debreu, Brouwer, Kakutani, and Hartman--Stampacchia type results. These connections motivate our use of Sperner not as an opponent to be disproved, but as the natural discrete benchmark against which an exact continuous formulation can be evaluated.

From a computational viewpoint, a triangulation introduces a resolution parameter. This is standard in simplicial algorithms and partition-based optimization: increasing resolution improves localization but also increases the number of cells and vertices that must be processed~\cite{EdelsbrunnerGrayson,GToth}. Our question is therefore specific. When preferences themselves admit a compact parametric representation, can one recover the entire fair-allocation region directly, rather than sampling the preference map on a successively finer mesh? The construction below answers this question for three agents with consistent reciprocal preferences.

\subsection{Pairwise comparisons and preference elicitation}
Pairwise-comparison matrices provide the decision-theoretic ingredient. Their modern use is closely associated with ratio-scale preference modelling and the Analytic Hierarchy Process~\cite{Saaty1977,Saaty1980,Saaty1986}. The broader decision-analysis literature has examined their interpretation, axiomatic properties, and practical use~\cite{Brunelli,FormanGass,Mazurek2023}. Consistency and inconsistency are central because pairwise judgments need not satisfy exact multiplicative transitivity; methods for analysing, repairing, or completing such matrices include~\cite{FedrizziBrunelliCaprila,FedrizziGiove,MazurekLinares}. Pairwise judgments have also been embedded in optimization-based elicitation models~\cite{SirajMikhailovKeane}, and empirical comparisons with more direct optimization formulations illustrate their role as a practical interface for decision support~\cite{AbelGalpinPatonKeane}.

We use only the part of this machinery that is needed for the continuous allocation model. In particular, we do not use reciprocal matrices to construct a final ranking through an eigenvector algorithm. We assume exact multiplicative consistency and exploit the induced ratio representation directly. This yields two independent thresholds per agent and, crucially, turns those thresholds into linear indifference boundaries on the allocation simplex. Inconsistent matrices are therefore outside the formal results below and are left for future work.

\section[Model formulation]{Model formulation}\label{sec:model}

\subsection{Allocation simplex and reciprocal threshold model}
There are three agents, indexed by $i\in\{1,2,3\}$, and three alternatives $A,B,C$. The feasible allocation space is the closed simplex
\begin{equation}\label{eq:T}
T=\{(a,b,c)\in\R^3_{\ge0}:a+b+c=1\}.
\end{equation}
The coordinates $a,b,c$ are the values attached to alternatives $A,B,C$, respectively.

For each agent $i$, preferences are represented by the positive reciprocal matrix
\begin{equation}\label{eq:M}
M_i=\begin{pmatrix}1&\lam{i}{AB}&\lam{i}{AC}\\[2pt] 1/\lam{i}{AB}&1&\lam{i}{BC}\\[2pt] 1/\lam{i}{AC}&1/\lam{i}{BC}&1\end{pmatrix},\qquad \lam{i}{XY}>0.
\end{equation}
The parameter $\lam{i}{XY}$ is the pairwise threshold of agent $i$ between the corresponding alternatives; operationally, $\lam{i}{AB}$ is the ratio $a/b$ at which agent $i$ is indifferent between $A$ and $B$, and similarly for the other two pairs. Reciprocity means that the reverse comparison contains no new information. Throughout the paper, $M_i$ is assumed to be multiplicatively consistent. Hence
\begin{equation}\label{eq:consistency}
\lam{i}{AC}=\lam{i}{AB}\lam{i}{BC},\quad\text{equivalently}\quad \lam{i}{BC}=\frac{\lam{i}{AC}}{\lam{i}{AB}}.
\end{equation}
Equivalently, there exists a positive vector $w^i=(w^i_A,w^i_B,w^i_C)$, unique up to multiplication by a positive constant, for which
\begin{equation}\label{eq:ratios}
\lam{i}{AB}=\frac{w^i_A}{w^i_B},\qquad\lam{i}{AC}=\frac{w^i_A}{w^i_C},\qquad\lam{i}{BC}=\frac{w^i_B}{w^i_C}.
\end{equation}
We call $w^i$ the \emph{priority vector} of agent $i$. Thus only two thresholds are independent for each agent. This ratio representation, rather than the spectral theory of reciprocal matrices, is the structural fact used throughout the remainder of the paper.

\subsection{Incentive and disincentive preference geometry}
The same reciprocal thresholds admit two monotonic interpretations of the coordinates of $T$.

\emph{Incentive model.} The coordinate attached to an alternative is a benefit: larger values make that alternative more attractive. For example, $A,B,C$ may be three job positions and $a,b,c$ their associated salaries or bonuses.

\emph{Disincentive model.} The coordinate attached to an alternative is a cost or burden: smaller values make that alternative more attractive. The rental-harmony problem is the standard example, with $A,B,C$ representing rooms and $a,b,c$ their rents.

For agent $i$, define the normalized scores
\begin{equation}\label{eq:scores}
s^i_A(a,b,c)=\frac{a}{w^i_A},\qquad s^i_B(a,b,c)=\frac{b}{w^i_B},\qquad s^i_C(a,b,c)=\frac{c}{w^i_C}.
\end{equation}
In the incentive model, the preferred alternative is the one with the largest normalized score; in the disincentive model, it is the one with the smallest normalized score. The priority vector is therefore not a vector of utilities. In the incentive model a larger $w^i_A$ makes $A$ \emph{less} attractive, because a larger benefit $a$ is needed for $A$ to be preferred; in the disincentive model a larger $w^i_A$ makes $A$ \emph{more} attractive, because a larger cost $a$ is accepted for it. For instance, if $\lam{i}{AB}=2$, then in the incentive model agent $i$ prefers $A$ to $B$ only if $a>2b$, whereas in the disincentive model she prefers $A$ to $B$ as long as $a<2b$. Pairwise equality of the scores gives the common indifference boundaries
\begin{equation}\label{eq:lines}
a=\lam{i}{AB}b,\qquad a=\lam{i}{AC}c,\qquad b=\lam{i}{BC}c.
\end{equation}
By \eqref{eq:consistency}, the three lines are compatible and meet in a unique interior point.

The strict preference regions in the incentive model are
\begin{align*}
P^{+}_{i,A}&=\{(a,b,c)\in T: a>\lam{i}{AB}b,\ a>\lam{i}{AC}c\},\\
P^{+}_{i,B}&=\{(a,b,c)\in T: b>a/\lam{i}{AB},\ b>\lam{i}{BC}c\},\\
P^{+}_{i,C}&=\{(a,b,c)\in T: c>a/\lam{i}{AC},\ c>b/\lam{i}{BC}\}.
\end{align*}
For the disincentive model, the same boundaries apply with all inequalities reversed:
\begin{align*}
P^{-}_{i,A}&=\{(a,b,c)\in T: a<\lam{i}{AB}b,\ a<\lam{i}{AC}c\},\\
P^{-}_{i,B}&=\{(a,b,c)\in T: b<a/\lam{i}{AB},\ b<\lam{i}{BC}c\},\\
P^{-}_{i,C}&=\{(a,b,c)\in T: c<a/\lam{i}{AC},\ c<b/\lam{i}{BC}\}.
\end{align*}
Both models share the same three pairwise indifference lines, which are the cevians of $T$ through $P^i$ (Definitions~\ref{def:loci} and \ref{def:center} below). The incentive partition is formed by the three half-segments joining $P^i$ to the edges of $T$, whereas the disincentive partition is formed by the three half-segments joining $P^i$ to the vertices of $T$. Figure~\ref{fig:agent} illustrates both partitions.

\begin{figure}[ht]
\centering
\includegraphics[width=0.92\textwidth]{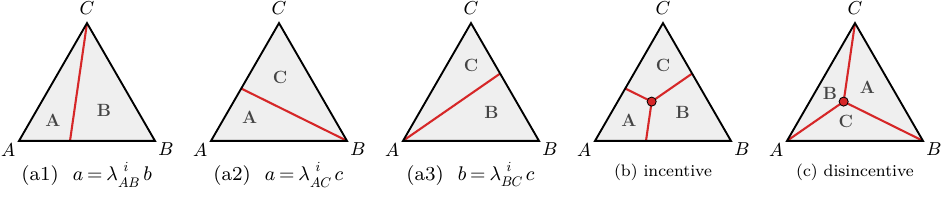}
\caption{Individual-agent preference geometry. Panels (a1)–(a3) show the three pairwise indifference segments induced by the reciprocal thresholds of agent $i$, with the sides labelled for the incentive model (the labels are exchanged in the disincentive model); panel (b) shows the preference partition obtained by superimposing them in the incentive model, and panel (c) the partition in the disincentive model.}
\label{fig:agent}
\end{figure}

\begin{definition}\label{def:loci}
For agent $i$, the three pairwise indifference segments are
\[
\{(a,b,c)\in T:a=\lam{i}{AB}b\},\qquad \{(a,b,c)\in T:a=\lam{i}{AC}c\},\qquad \{(a,b,c)\in T:b=\lam{i}{BC}c\}.
\]
\end{definition}
Each segment is a cevian of $T$.

\begin{definition}\label{def:center}
The \emph{indifference centre} $P^i$ is the unique point at which agent $i$ is simultaneously indifferent among $A,B,C$. If $D_i=\lam{i}{AB}+\lam{i}{AC}+\lam{i}{AB}\lam{i}{AC}$, then
\begin{equation}\label{eq:center}
P^i=\left(\frac{\lam{i}{AB}\lam{i}{AC}}{D_i},\frac{\lam{i}{AC}}{D_i},\frac{\lam{i}{AB}}{D_i}\right).
\end{equation}
Equivalently, $P^i$ is the normalization of $w^i$ to unit sum.
\end{definition}

\subsection{Preference regions and solution components}
Superimposing the three agent-specific partitions produces a finite collection of cells. At a point of $T$ where no agent is indifferent, the \emph{preference profile} is the triple $(X_1,X_2,X_3)$, where $X_i$ is the alternative strictly preferred ($A$, $B$ or $C$) by agent $i$. A \emph{preference region} is a maximal connected subset of $T$ on which the preference profile is constant. Because each preference profile is described by finitely many strict linear inequalities, every nonempty preference region is convex.

For example, in the incentive model, the region on which agent 1 prefers $A$, agent 2 prefers $B$, and agent 3 prefers $C$ is
\[
\left\{(a,b,c)\in T: a>\lam{1}{AB}b,\ a>\lam{1}{AC}c,\ b>a/\lam{2}{AB},\ b>\lam{2}{BC}c,\ c>a/\lam{3}{AC},\ c>b/\lam{3}{BC}\right\}.
\]
Figure~\ref{fig:partition} shows the superposition of three partitions in the incentive model.

\begin{figure}[ht]
\centering
\includegraphics[width=0.48\textwidth]{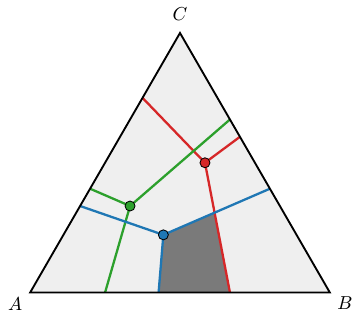}
\caption{A possible partition of $T$ induced by the preferences of the three agents in the incentive model. Red, green, and blue segments correspond to agents 1, 2, and 3, respectively. The shaded cell, where agent 1 prefers $A$ and agents 2 and 3 prefer $B$, is a boundary region.}
\label{fig:partition}
\end{figure}

\begin{definition}\label{def:boundary}
A preference region $R$ is a \emph{boundary region} if $\overline{R}\cap\partial T\neq\emptyset$.
\end{definition}
\begin{definition}\label{def:interior}
A preference region $R$ is an \emph{interior region} if $\overline{R}\subset\operatorname{int}(T)$.
\end{definition}

The fair allocations studied in this paper are those at which the three agents strictly prefer three different alternatives. For a permutation $(i,j,k)$ of $(1,2,3)$, the assignment $i\mapsto A$, $j\mapsto B$, $k\mapsto C$ therefore determines the following solution component.

\begin{definition}\label{def:S}
In the incentive model, define
\[
\Sp^{+}_{ijk}=\left\{(a,b,c)\in T: a>\lam{i}{AB}b,\ a>\lam{i}{AC}c,\ b>a/\lam{j}{AB},\ b>\lam{j}{BC}c,\ c>a/\lam{k}{AC},\ c>b/\lam{k}{BC}\right\}.
\]
In the disincentive model, define
\[
\Sp^{-}_{ijk}=\left\{(a,b,c)\in T: a<\lam{i}{AB}b,\ a<\lam{i}{AC}c,\ b<a/\lam{j}{AB},\ b<\lam{j}{BC}c,\ c<a/\lam{k}{AC},\ c<b/\lam{k}{BC}\right\}.
\]
The overall strict solution regions are
\[
\Sp^{+}=\bigcup_{(i,j,k)\in S_3}\Sp^{+}_{ijk},\qquad \Sp^{-}=\bigcup_{(i,j,k)\in S_3}\Sp^{-}_{ijk}.
\]
and we write $\mathcal S^\pm_\pi:=\mathcal S^\pm_{ijk}$ for the corresponding bijection $\pi$.
\end{definition}

The strict solution components exclude allocations at which one or more agents are indifferent. To include such cases, the strict inequalities are replaced by weak inequalities.

\begin{definition}\label{def:E}
For a permutation $(i,j,k)$ of $(1,2,3)$, define
\begin{align*}
\Ex^{+}_{ijk}&=\left\{(a,b,c)\in T: a\ge\lam{i}{AB}b,\ a\ge\lam{i}{AC}c,\ b\ge a/\lam{j}{AB},\ b\ge\lam{j}{BC}c,\ c\ge a/\lam{k}{AC},\ c\ge b/\lam{k}{BC}\right\},\\
\Ex^{-}_{ijk}&=\left\{(a,b,c)\in T: a\le\lam{i}{AB}b,\ a\le\lam{i}{AC}c,\ b\le a/\lam{j}{AB},\ b\le\lam{j}{BC}c,\ c\le a/\lam{k}{AC},\ c\le b/\lam{k}{BC}\right\},
\end{align*}
and write $\Ex^{\pm}_{\pi}:=\Ex^{\pm}_{ijk}$ for the corresponding bijection $\pi$. The overall extended solution sets are $\Ex^{+}=\bigcup_{(i,j,k)\in S_3}\Ex^{+}_{ijk}$ and $\Ex^{-}=\bigcup_{(i,j,k)\in S_3}\Ex^{-}_{ijk}$.
\end{definition}

\section{Exact characterization of the continuous solution region}\label{sec:main}

We now turn from model construction to the exact mathematical characterization of the strict and extended solution sets. The first step is to decide when a fixed assignment of agents to alternatives generates a nonempty solution component.

\subsection{Existence of a solution component}
The key simplification is to work with the ratios $b/a$ and $c/a$ rather than directly with barycentric coordinates.

\begin{proposition}\label{prop:exist}
Let $(i,j,k)$ be a permutation of $(1,2,3)$. In the incentive model, $\Sp^{+}_{ijk}\neq\emptyset$ if and only if
\begin{equation}\label{eq:inc5}
\lam{k}{AC}>\lam{i}{AC},\quad \lam{j}{AB}>\lam{i}{AB},\quad \lam{k}{BC}>\lam{j}{BC},\quad \lam{k}{AC}>\lam{i}{AB}\lam{j}{BC},\quad \lam{i}{AC}<\lam{j}{AB}\lam{k}{BC}.
\end{equation}
In the disincentive model, $\Sp^{-}_{ijk}\neq\emptyset$ if and only if
\begin{equation}\label{eq:dis5}
\lam{k}{AC}<\lam{i}{AC},\quad \lam{j}{AB}<\lam{i}{AB},\quad \lam{k}{BC}<\lam{j}{BC},\quad \lam{k}{AC}<\lam{i}{AB}\lam{j}{BC},\quad \lam{i}{AC}>\lam{j}{AB}\lam{k}{BC}.
\end{equation}
Consequently, the overall solution region is nonempty if and only if at least one permutation satisfies the corresponding five inequalities.
\end{proposition}

\begin{proof}
Every point of either solution component has $a,b,c>0$. Set $x=b/a$, $y=c/a$. Moreover, each pair $(x,y)\in\R^2_{>0}$ determines a unique point of $\operatorname{int}(T)$ through
\begin{equation}\label{eq:xy}
(a,b,c)=\frac{1}{1+x+y}(1,x,y).
\end{equation}

\emph{Incentive model.} The six inequalities defining $\Sp^{+}_{ijk}$ are equivalent to
\begin{equation}\label{eq:incxy}
\frac{1}{\lam{j}{AB}}<x<\frac{1}{\lam{i}{AB}},\qquad \max\left\{\frac{1}{\lam{k}{AC}},\frac{x}{\lam{k}{BC}}\right\}<y<\min\left\{\frac{1}{\lam{i}{AC}},\frac{x}{\lam{j}{BC}}\right\}.
\end{equation}
Assume first that $\Sp^{+}_{ijk}\neq\emptyset$. The interval for $x$ gives $\lam{j}{AB}>\lam{i}{AB}$. Comparing the constant bounds for $y$ gives $\lam{k}{AC}>\lam{i}{AC}$, while comparing the bounds proportional to $x$ gives $\lam{k}{BC}>\lam{j}{BC}$. Moreover,
\[
\frac{1}{\lam{k}{AC}}<\frac{x}{\lam{j}{BC}}<\frac{1}{\lam{i}{AB}\lam{j}{BC}},
\]
which implies $\lam{k}{AC}>\lam{i}{AB}\lam{j}{BC}$. Likewise,
\[
\frac{1}{\lam{j}{AB}\lam{k}{BC}}<\frac{x}{\lam{k}{BC}}<\frac{1}{\lam{i}{AC}},
\]
which implies $\lam{i}{AC}<\lam{j}{AB}\lam{k}{BC}$. Hence the five inequalities in \eqref{eq:inc5} are necessary.

Conversely, assume \eqref{eq:inc5}. Define
\[
L=\max\left\{\frac{1}{\lam{j}{AB}},\frac{\lam{j}{BC}}{\lam{k}{AC}}\right\},\qquad U=\min\left\{\frac{1}{\lam{i}{AB}},\frac{\lam{k}{BC}}{\lam{i}{AC}}\right\}.
\]
The five inequalities in \eqref{eq:inc5} imply $L<U$. Indeed, the four comparisons between the two lower and the two upper endpoints follow from $\lam{j}{AB}>\lam{i}{AB}$, $\lam{i}{AC}<\lam{j}{AB}\lam{k}{BC}$, $\lam{k}{AC}>\lam{i}{AB}\lam{j}{BC}$, and from the pair $\lam{k}{AC}>\lam{i}{AC}$, $\lam{k}{BC}>\lam{j}{BC}$. Choose $x\in(L,U)$. Then
\[
\max\left\{\frac{1}{\lam{k}{AC}},\frac{x}{\lam{k}{BC}}\right\}<\min\left\{\frac{1}{\lam{i}{AC}},\frac{x}{\lam{j}{BC}}\right\},
\]
so we may choose a positive $y$ strictly between these bounds. Equation \eqref{eq:xy} then gives a point of $\Sp^{+}_{ijk}$.

\emph{Disincentive model.} The six inequalities defining $\Sp^{-}_{ijk}$ are equivalent to
\begin{equation}\label{eq:disxy}
\frac{1}{\lam{i}{AB}}<x<\frac{1}{\lam{j}{AB}},\qquad \max\left\{\frac{1}{\lam{i}{AC}},\frac{x}{\lam{j}{BC}}\right\}<y<\min\left\{\frac{1}{\lam{k}{AC}},\frac{x}{\lam{k}{BC}}\right\}.
\end{equation}

If $\mathcal S^-_{ijk}\neq\emptyset$, the interval for $x$ gives $\lambda^j_{AB}<\lambda^i_{AB}$; comparing the constant bounds for $y$ gives $\lambda^k_{AC}<\lambda^i_{AC}$, and comparing the bounds proportional to $x$ gives $\lambda^k_{BC}<\lambda^j_{BC}$. Moreover, $1/\lambda^i_{AC}<x/\lambda^k_{BC}<1/(\lambda^j_{AB}\lambda^k_{BC})$ gives $\lambda^i_{AC}>\lambda^j_{AB}\lambda^k_{BC}$, and $1/(\lambda^i_{AB}\lambda^j_{BC})<x/\lambda^j_{BC}<1/\lambda^k_{AC}$ gives $\lambda^k_{AC}<\lambda^i_{AB}\lambda^j_{BC}$. Conversely, assume \eqref{eq:dis5} and set $L'=\max\{1/\lambda^i_{AB},\,\lambda^k_{BC}/\lambda^i_{AC}\}$, $U'=\min\{1/\lambda^j_{AB},\,\lambda^j_{BC}/\lambda^k_{AC}\}$. The inequalities \eqref{eq:dis5} give $L'<U'$, and for $x\in(L',U')$ we can choose $y$ strictly between the bounds in \eqref{eq:disxy}. Normalization through \eqref{eq:xy} completes the proof.

\end{proof}

\subsection{Assignment-score characterization}
The threshold inequalities admit a compact interpretation in terms of the consistent priority vectors. Let $\Pi$ be the six bijections from the agents to $\{A,B,C\}$ and define
\begin{equation}\label{eq:Phi}
\Phi(\pi)=\prod_{r=1}^{3}w^r_{\pi(r)},\qquad \pi\in\Pi.
\end{equation}
Although each $w^r$ is defined only up to a positive factor, the ordering of the six values $\Phi(\pi)$ is well defined: rescaling $w^r$ multiplies every assignment score by the same factor.

\begin{proposition}\label{prop:Phi}
Let $\pi_{ijk}$ denote the assignment $i\mapsto A$, $j\mapsto B$, $k\mapsto C$. Then
\begin{align}
\Sp^{+}_{ijk}\neq\emptyset&\iff \Phi(\pi_{ijk})<\Phi(\sigma)\ \text{for every }\sigma\neq\pi_{ijk},\label{eq:Phimin}\\
\Sp^{-}_{ijk}\neq\emptyset&\iff \Phi(\pi_{ijk})>\Phi(\sigma)\ \text{for every }\sigma\neq\pi_{ijk}.\label{eq:Phimax}
\end{align}
Thus a feasible incentive assignment is exactly the unique minimizer of $\Phi$, while a feasible disincentive assignment is exactly the unique maximizer.
\end{proposition}

\begin{proof}
Compare $\Phi(\pi_{ijk})$ with the assignments obtained by the three transpositions. Using \eqref{eq:ratios}, these three comparisons are equivalent to the first three inequalities in \eqref{eq:inc5}.
For instance, if $\sigma$ exchanges the alternatives of $i$ and $k$ ($i\mapsto C$, $j\mapsto B$, $k\mapsto A$), then

\[
\Phi(\pi_{ijk})<\Phi(\sigma)\iff w^i_Aw^j_Bw^k_C<w^i_Cw^j_Bw^k_A\iff \frac{w^i_A}{w^i_C}<\frac{w^k_A}{w^k_C}\iff \lam{i}{AC}<\lam{k}{AC}.
\]

Comparing $\Phi(\pi_{ijk})$ with the two 3-cycles gives the last two inequalities. For instance, the 3-cycle $\tau$: $i\mapsto B$, $j\mapsto C$, $k\mapsto A$ gives

\[
\Phi(\pi_{ijk})<\Phi(\tau)\iff w^i_Aw^j_Bw^k_C<w^i_Bw^j_Cw^k_A\iff \frac{w^i_A}{w^i_B}\cdot\frac{w^j_B}{w^j_C}<\frac{w^k_A}{w^k_C}\iff \lam{i}{AB}\lam{j}{BC}<\lam{k}{AC}.
\]

Proposition~\ref{prop:exist} therefore yields \eqref{eq:Phimin}. Reversing all five comparisons gives \eqref{eq:Phimax} for the disincentive model.
\end{proof}

\subsection{Interior characterization of the solution region}
The next proposition links the combinatorial requirement that the three agents select different alternatives to a simple geometric property of the induced cell.

\begin{proposition}\label{prop:interior}
A preference region $R$ is contained in the solution region if and only if $R$ is an interior region. Equivalently, the interior preference regions are precisely the nonempty solution components $\Sp^{\pm}_{ijk}$.
\end{proposition}

\begin{proof}
Suppose first that $R$ is not contained in the solution region. Since there are three agents and three alternatives, at least one alternative is absent from the strict preference profile of $R$. Without loss of generality, suppose that no agent prefers $C$.

\emph{Incentive model.} Take $(a,b,c)\in R$ and, for $t\in(0,1]$, define
\[
p_t=\frac{(a,b,tc)}{a+b+tc}.
\]
Along this path the ratio $a/b$ is unchanged, while $C$ becomes progressively less attractive. Hence every strict comparison between $A$ and $B$ is preserved and every strict preference of $A$ or $B$ over $C$ remains valid. Therefore $p_t\in R$ for all $t>0$. As $t\downarrow0$,
\[
p_t\longrightarrow\left(\frac{a}{a+b},\frac{b}{a+b},0\right)\in\partial T.
\]
Thus $R$ is a boundary region.

\emph{Disincentive model.} Again every agent prefers either $A$ or $B$. Define
\[
q_t=(ta,tb,1-t(a+b)),\qquad t\in(0,1].
\]
The ratio $a/b$ is preserved, while $A$ and $B$ become smaller relative to $C$. Since smaller normalized scores are preferred in the disincentive model, every strict preference for $A$ or $B$ remains valid. Hence $q_t\in R$ for all $t>0$, and $q_t\to(0,0,1)\in\partial T$ as $t\downarrow0$. Thus $R$ is again a boundary region. The same construction applies cyclically if the missing alternative is $A$ or $B$. Therefore every interior preference region is contained in the solution region.

Conversely, let $R=\Sp^{\pm}_{ijk}$ be a nonempty solution component. We show that its closure does not meet any edge of $T$.

In the incentive model, suppose that a sequence $(a_n,b_n,c_n)\in R$ converges to the edge $c=0$. Since agent $k$ strictly prefers $C$, $a_n<\lam{k}{AC}c_n$ and $b_n<\lam{k}{BC}c_n$. Thus $a_n,b_n\to0$ whenever $c_n\to0$, contradicting $a_n+b_n+c_n=1$. The edges $a=0$ and $b=0$ are ruled out analogously by using the agents who prefer $A$ and $B$.

In the disincentive model, if $c_n\to0$, the agents who prefer $A$ and $B$ satisfy $a_n<\lam{i}{AC}c_n$ and $b_n<\lam{j}{BC}c_n$, so again $a_n,b_n\to0$, a contradiction. The remaining edges are excluded cyclically. Hence $\overline{R}\subset\operatorname{int}(T)$ in both models, and $R$ is an interior region.
\end{proof}

\subsection{Uniqueness of the solution region}
The score characterization in \eqref{eq:Phimin}--\eqref{eq:Phimax} removes the need for a case-by-case analysis of the six possible assignments.

\begin{proposition}\label{prop:unique}
For each model, there is at most one permutation $(i,j,k)$ of $\{1,2,3\}$ for which the corresponding solution component $\Sp^{\pm}_{ijk}$ is nonempty. Consequently, the solution region is either empty or consists of a single interior preference region.
\end{proposition}

\begin{proof}
In the incentive model, if $\Sp^{+}_{ijk}$ is nonempty, then \eqref{eq:Phimin} shows that $\pi_{ijk}$ is the unique minimizer of $\Phi$ over the six assignments. Two distinct solution components would therefore require two distinct strict global minimizers, which is impossible. In the disincentive model, \eqref{eq:Phimax} shows instead that a feasible assignment is the unique maximizer of $\Phi$. Again, two distinct solution components are impossible. The final statement follows from Proposition~\ref{prop:interior}.
\end{proof}

\begin{theorem}\label{thm:main}
Assume that the three agents are represented by multiplicatively consistent reciprocal preference matrices. For every assignment $\pi\in\Pi$,
\[
\Sp^{+}_{\pi}\neq\emptyset\iff\pi\text{ is the unique minimizer of }\Phi,
\qquad
\Sp^{-}_{\pi}\neq\emptyset\iff\pi\text{ is the unique maximizer of }\Phi.
\]
Consequently, in either model the strict solution region is either empty or consists of a unique convex preference region whose closure is contained in $\operatorname{int}(T)$.
\end{theorem}

\begin{proof}
The minimizer/maximizer characterization is Proposition~\ref{prop:Phi}. Proposition~\ref{prop:unique} gives uniqueness of the nonempty component, and Proposition~\ref{prop:interior} shows that such a component is an interior preference region. Convexity follows from the half-space representation in Definition~\ref{def:S}.
\end{proof}

\section{Degenerate configurations}\label{sec:degenerate}

We now study the extended solution sets $\Ex^{\pm}_{ijk}$ of Definition~\ref{def:E}, which also contain allocations at which some agents are indifferent.

\begin{remark}\label{rem:closure}
If $\Sp^{\pm}_{ijk}$ is nonempty, then $\Ex^{\pm}_{ijk}$ is its closure in $T$. If the strict component is empty, however, the corresponding extended component may still be nonempty and may reduce to a line segment or a single point.
\end{remark}

\begin{proposition}\label{prop:weak}
For every assignment $\pi\in\Pi$,
\[
\Ex^{+}_{\pi}\neq\emptyset\iff\pi\text{ is a minimizer of }\Phi,\qquad \Ex^{-}_{\pi}\neq\emptyset\iff\pi\text{ is a maximizer of }\Phi.
\]
Consequently:
\begin{enumerate}[label=(\roman*)]
\item $\Ex^{+}$ and $\Ex^{-}$ are always nonempty;
\item $\Sp^{+}=\emptyset$ (respectively $\Sp^{-}=\emptyset$) if and only if the minimum (respectively maximum) of $\Phi$ is attained by at least two assignments;
\item if $\Sp^{+}=\emptyset$ (respectively $\Sp^{-}=\emptyset$), then $\Ex^{\pm}$ is the union of the sets $\Ex^{\pm}_{\pi}$ over the optimal assignments $\pi$, and each of them is a segment or a single point.
\end{enumerate}
\end{proposition}

\begin{proof}
A direct proof of the equivalences repeats the proof of Proposition~\ref{prop:exist} with weak inequalities, which turns \eqref{eq:inc5} into the five weak comparisons expressing $\Phi(\pi)\le\Phi(\sigma)$ for all $\sigma$ (Proposition~\ref{prop:Phi}). Item (i) follows because $\Phi$ attains its minimum and its maximum on the finite set $\Pi$, and item (ii) is Theorem~\ref{thm:main}. For (iii), if $\Ex^{\pm}_{\pi}$ had nonempty interior, its interior points would satisfy the strict inequalities, so $\Sp^{\pm}_{\pi}\neq\emptyset$ and $\pi$ would be the unique optimizer, a contradiction. Thus $\Ex^{\pm}_{\pi}$ is a compact convex subset of $T$ with empty interior, hence contained in a line, hence a segment or a point.
\end{proof}

We now describe the main degenerate configurations; by Proposition~\ref{prop:weak}, each of them corresponds to a tie in the optimum of $\Phi$. They are expressed directly in terms of the threshold parameters because this form is convenient for identifying coincident boundaries and indifference centres.

\textbf{Degeneracies generated by equal pairwise thresholds.} The first three strict inequalities in Proposition~\ref{prop:exist} fail in several important configurations.
\begin{enumerate}[label=(\alph*),leftmargin=1.6em]
\item If one of
\[
\lam{1}{AB}=\lam{2}{AB}=\lam{3}{AB},\qquad \lam{1}{AC}=\lam{2}{AC}=\lam{3}{AC},\qquad \lam{1}{BC}=\lam{2}{BC}=\lam{3}{BC}
\]
holds, all three agents share the same threshold for the corresponding pair of alternatives. If exactly one equality holds, then $\Sp^{\pm}=\emptyset$ and, by Proposition~\ref{prop:weak}(iii), the extended solution set is either a segment or a point; Figure~\ref{fig:deg1} shows segment-valued cases in panels (a) and (d) and point-valued cases in panels (b) and (e). When it is a single point, two agents have the same priority vector, so the configuration is also of type (b). If two of the equalities hold, consistency forces the third one as well. In that case $P^1=P^2=P^3$, and the extended solution set reduces to the common indifference centre; see panels (c) and (f) of the same figure.

\begin{figure}[ht]
\centering
\includegraphics[width=0.78\textwidth]{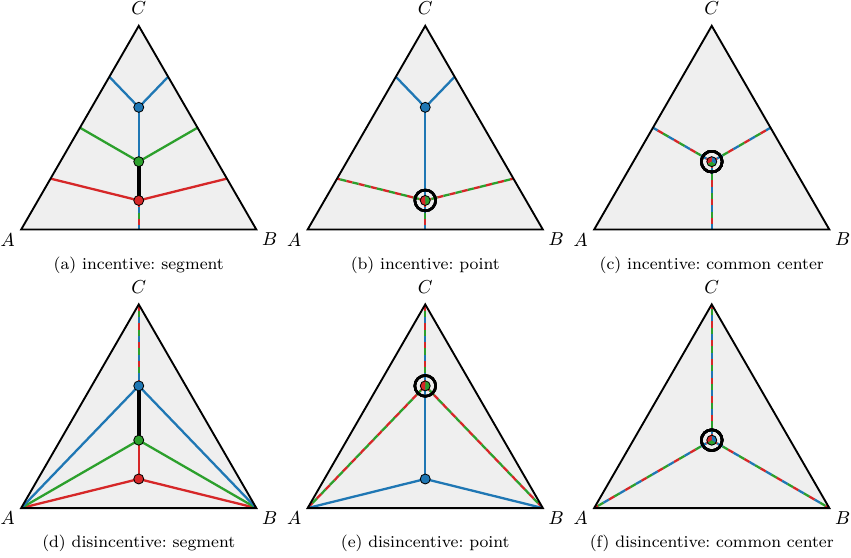}
\caption{Degenerate configurations generated by a pairwise threshold common to all agents ($\lambda_{AB}=1$). Top row: incentive model; bottom row: disincentive model. (a), (d): the extended solution set is a segment. (b), (e): it reduces to a point; two agents then have the same priority vector, so this configuration is also of type (b). (c), (f): all thresholds coincide and the three indifference centres collapse to a single point, which is the extended solution set. The extended solution set is drawn in black (thick segment or circled point); where the boundaries of several agents coincide, their colours alternate along the common piece, and coincident indifference centres are drawn as a disc split into the colours of the agents involved.}
\label{fig:deg1}
\end{figure}

\item If, for two distinct agents $r$ and $s$, $\lam{r}{AB}=\lam{s}{AB}$, $\lam{r}{AC}=\lam{s}{AC}$, $\lam{r}{BC}=\lam{s}{BC}$, then $P^r=P^s$. In the configurations considered here, the extended solution set is either a segment or a point; see Figure~\ref{fig:deg2}, with segment-valued cases in panels (a) and (c) and point-valued cases in panels (b) and (d). In the point-valued examples the third agent also shares one threshold with the other two, so these configurations are also of type (a), and the point is the common centre $P^r=P^s$.

\begin{figure}[ht]
\centering
\includegraphics[width=0.78\textwidth]{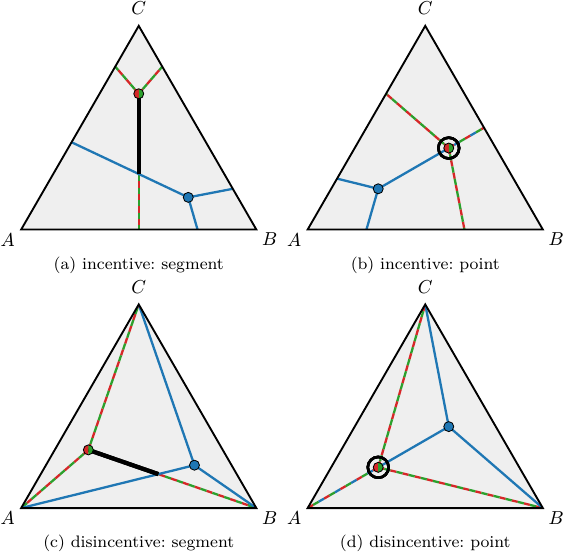}
\caption{Degenerate configurations with two coincident indifference centres (agents 1 and 2 identical). Top row: incentive model; bottom row: disincentive model. Panels (a) and (c) show segment-valued extended solution sets, and panels (b) and (d) point-valued ones.}
\label{fig:deg2}
\end{figure}

\item A further family of failures occurs when two agents share one threshold and the remaining threshold ordering prevents any permutation from satisfying Proposition~\ref{prop:exist}. For the incentive model, the following six configurations are sufficient for $\Sp^{+}=\emptyset$:
\begin{enumerate}[label=(\roman*)]
\item $\lam{r}{AB}=\lam{s}{AB}>\lam{t}{AB}$, $\lam{t}{AC}=\max\{\lam{r}{AC},\lam{s}{AC},\lam{t}{AC}\}$;
\item $\lam{r}{AB}=\lam{s}{AB}<\lam{t}{AB}$, $\lam{t}{BC}=\max\{\lam{r}{BC},\lam{s}{BC},\lam{t}{BC}\}$;
\item $\lam{r}{AC}=\lam{s}{AC}>\lam{t}{AC}$, $\lam{t}{AB}=\max\{\lam{r}{AB},\lam{s}{AB},\lam{t}{AB}\}$;
\item $\lam{r}{AC}=\lam{s}{AC}<\lam{t}{AC}$, $\lam{t}{BC}=\min\{\lam{r}{BC},\lam{s}{BC},\lam{t}{BC}\}$;
\item $\lam{r}{BC}=\lam{s}{BC}>\lam{t}{BC}$, $\lam{t}{AB}=\min\{\lam{r}{AB},\lam{s}{AB},\lam{t}{AB}\}$;
\item $\lam{r}{BC}=\lam{s}{BC}<\lam{t}{BC}$, $\lam{t}{AC}=\min\{\lam{r}{AC},\lam{s}{AC},\lam{t}{AC}\}$,
\end{enumerate}
where $\{r,s,t\}$ is a permutation of $\{1,2,3\}$. For the disincentive model, the corresponding six configurations are obtained by reversing the order relations and interchanging maxima and minima:
\begin{enumerate}[label=(\roman*)]
\item $\lam{r}{AB}=\lam{s}{AB}<\lam{t}{AB}$, $\lam{t}{AC}=\min\{\lam{r}{AC},\lam{s}{AC},\lam{t}{AC}\}$;
\item $\lam{r}{AB}=\lam{s}{AB}>\lam{t}{AB}$, $\lam{t}{BC}=\min\{\lam{r}{BC},\lam{s}{BC},\lam{t}{BC}\}$;
\item $\lam{r}{AC}=\lam{s}{AC}<\lam{t}{AC}$, $\lam{t}{AB}=\min\{\lam{r}{AB},\lam{s}{AB},\lam{t}{AB}\}$;
\item $\lam{r}{AC}=\lam{s}{AC}>\lam{t}{AC}$, $\lam{t}{BC}=\max\{\lam{r}{BC},\lam{s}{BC},\lam{t}{BC}\}$;
\item $\lam{r}{BC}=\lam{s}{BC}<\lam{t}{BC}$, $\lam{t}{AB}=\max\{\lam{r}{AB},\lam{s}{AB},\lam{t}{AB}\}$;
\item $\lam{r}{BC}=\lam{s}{BC}>\lam{t}{BC}$, $\lam{t}{AC}=\max\{\lam{r}{AC},\lam{s}{AC},\lam{t}{AC}\}$,
\end{enumerate}
again with $\{r,s,t\}$ a permutation of $\{1,2,3\}$.

To illustrate why these configurations prevent strict feasibility, consider the first incentive case. Since $\lam{r}{AB}=\lam{s}{AB}>\lam{t}{AB}$, the condition $\lam{j}{AB}>\lam{i}{AB}$ forces $i=t$. But then $\lam{t}{AC}=\max\{\lam{r}{AC},\lam{s}{AC},\lam{t}{AC}\}$ contradicts $\lam{k}{AC}>\lam{i}{AC}$. The remaining cases are checked analogously. Representative point- and segment-valued extended solutions are grouped in Figure~\ref{fig:deg3}, panels (a) and (c) correspond to configuration (i) of the respective model and panels (b) and (d) to configuration (ii).

\end{enumerate}

\begin{figure}[ht]
\centering
\includegraphics[width=0.78\textwidth]{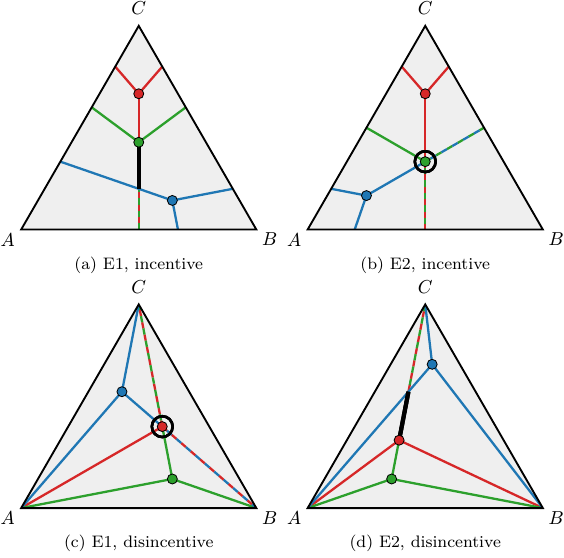}
\caption{Representative threshold-order degeneracies preventing strict feasibility. Top row: incentive model; bottom row: disincentive model. E1 denotes configuration (i) and E2 configuration (ii) of the respective model, with $r,s$ agents 1 and 2 and $t$ agent 3.}
\label{fig:deg3}
\end{figure}

\textbf{Degeneracies generated by cycle equalities.} A different failure of strict feasibility occurs when one of the two cycle inequalities in Proposition~\ref{prop:exist} becomes an equality. In the incentive model, one such configuration is
\begin{equation}\label{eq:cycinc}
\lam{r}{AC}=\lam{s}{AB}\lam{t}{BC},\qquad \lam{t}{AB}>\lam{s}{AB}>\lam{r}{AB},\qquad \lam{s}{AC}<\min\{\lam{r}{AC},\lam{t}{AC}\},
\end{equation}
where $\{r,s,t\}$ is a permutation of $\{1,2,3\}$. In the disincentive model, the corresponding configuration is
\begin{equation}\label{eq:cycdis}
\lam{r}{AC}=\lam{s}{AB}\lam{t}{BC},\qquad \lam{t}{AB}<\lam{s}{AB}<\lam{r}{AB},\qquad \lam{s}{AC}>\max\{\lam{r}{AC},\lam{t}{AC}\}.
\end{equation}
For example, in the incentive model, the equality $\lam{r}{AC}=\lam{s}{AB}\lam{t}{BC}$ rules out the identifications $(r,s,t)=(k,i,j)$ and $(r,s,t)=(i,j,k)$ if both cycle inequalities in Proposition~\ref{prop:exist} are to be strict. The additional order restrictions rule out the remaining permutations. The disincentive argument is analogous. In both models the extended solution set reduces to a single point; the incentive and disincentive cases are shown side by side in Figure~\ref{fig:deg4}.

\begin{figure}[ht]
\centering
\includegraphics[width=0.78\textwidth]{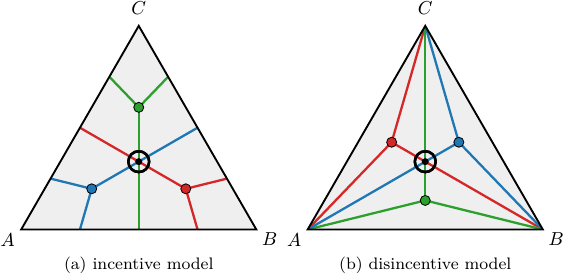}
\caption{Cycle-equality degeneracy. In both models the strict solution component disappears and the extended solution set reduces to a point.}
\label{fig:deg4}
\end{figure}

\section{Comparison with Sperner-based discretization and worked examples}\label{sec:sperner}

The model developed in Section~\ref{sec:model} provides a continuous description of the agents' preferences, and Sections~\ref{sec:main}--\ref{sec:degenerate} identify the exact solution region directly in the simplex $T$. A natural benchmark is the family of discrete procedures based on Sperner labellings. Such procedures replace the continuous simplex by a finite triangulation and query preferences only at its vertices; Sperner's lemma then guarantees the existence of at least one fully labelled elementary triangle whenever the corresponding boundary conditions are satisfied~\cite{Sperner,Su}; see Lemma~\ref{lem:sperner} below. Polytopal and combinatorial extensions are given in~\cite{DeLoeraPetersonSu,Meunier}.

We compare finite discretization procedures; Sperner's lemma itself is not at issue. A fully labelled triangle is a discrete certificate. By contrast, the reciprocal-matrix model yields an explicit continuous region described by inequalities. This distinction affects the interpretation of the output, the role of the triangulation, and the amount of preference information that must be elicited.

\subsection{A Sperner-based discretization of the preference problem}
Let $\T$ be a triangulation of $T$. Assume that every vertex $v$ of $\T$ is assigned to one of the three agents and that the ownership rule is chosen so that the three agents are represented in every elementary triangle. The owner of $v$ is asked which of the alternatives $A$, $B$, or $C$ is preferred at the allocation represented by $v$. The answer is used as the label of $v$; at an indifference point, any of the tied preferred alternatives may be selected.

An elementary triangle is called fully labelled if its three vertices receive the three different labels $A$, $B$, and $C$. A labelling of the vertices of $\T$ by $A,B,C$ is a \emph{Sperner labelling} if the three vertices of $T$ receive three distinct labels and every vertex of $\T$ lying on an edge of $T$ receives one of the two labels assigned to the endpoints of that edge (the \emph{Sperner boundary condition}). Under this condition, Sperner's lemma guarantees the existence of such an elementary triangle:

\begin{lemma}\label{esperner}(Sperner~\cite{Sperner}; see also~\cite{Su})\label{lem:sperner}
Every Sperner labelling of a triangulation of $T$ has an odd number of fully labelled elementary triangles; in particular, it has at least one.
\end{lemma}

Since the three agents are represented in every elementary triangle, a fully labelled triangle records three distinct local choices made by three different agents. The continuous model developed in Section~\ref{sec:model} can be used to generate these labels; the next subsection shows that the resulting labellings satisfy the Sperner boundary condition.

\subsection{Compatibility with Sperner labellings}
We write $v_A=(1,0,0)$, $v_B=(0,1,0)$, $v_C=(0,0,1)$ for the three vertices of $T$.

\begin{proposition}[Incentive model]\label{prop:spinc}
Consider the incentive model and label each vertex of $\T$ by an alternative preferred by the agent owning that vertex. Then the natural corner labelling $v_A\mapsto A$, $v_B\mapsto B$, $v_C\mapsto C$ satisfies Sperner's boundary condition. Consequently, the induced labelling is a Sperner labelling.
\end{proposition}

\begin{proof}
Consider, for example, the relative interior of the edge $c=0$. In the incentive model, the normalized score of alternative $C$ is $s^i_C(a,b,0)=0$. At least one of $a$ and $b$ is positive, so one of the scores $s^i_A$ or $s^i_B$ is strictly positive. Hence $C$ cannot be a preferred alternative at an interior point of this edge. Therefore only the labels $A$ and $B$, which are precisely the labels of the endpoints $v_A$ and $v_B$, can occur on $c=0$. The same argument applies cyclically to the edges $a=0$ and $b=0$.
\end{proof}

For the disincentive model the geometry is dual. On the edge where one coordinate vanishes, the corresponding alternative has the smallest normalized score and is therefore preferred. Thus the natural corner labelling is no longer compatible with the standard Sperner boundary convention. A simple cyclic relabelling resolves this issue.

\begin{proposition}[Disincentive model]\label{prop:spdis}
Consider the disincentive model. Let $\tau$ be a derangement of $\{A,B,C\}$, that is, a permutation with no fixed point. Resolve the ties at the three vertices of $T$ so that $v_A\mapsto\tau(A)$, $v_B\mapsto\tau(B)$, $v_C\mapsto\tau(C)$. At every other vertex of $\T$, use an alternative preferred by the owner of that vertex as its label. Then the resulting labelling satisfies Sperner's boundary condition relative to the corner labels $\tau(A),\tau(B),\tau(C)$.
\end{proposition}

\begin{proof}
Consider the relative interior of the edge $a=0$. For every agent $i$, $s^i_A(0,b,c)=0<s^i_B(0,b,c),\,s^i_C(0,b,c)$, so $A$ is the unique preferred alternative in the disincentive model. The edge $a=0$ is spanned by $v_B$ and $v_C$, whose corner labels are $\tau(B)$ and $\tau(C)$. Since $\tau$ is a derangement, $\tau(A)\neq A$; because $\tau$ is a permutation, this implies $A\in\{\tau(B),\tau(C)\}$. Hence the label $A$ is admissible on the edge $a=0$. The same argument applies cyclically to the other two edges.

At a corner, two coordinates are zero. Therefore the disincentive model produces a tie between the two alternatives associated with the zero coordinates. Since $\tau$ has no fixed points, $\tau(A)$ is one of the two alternatives tied at $v_A$, and similarly at $v_B$ and $v_C$. Thus the prescribed corner labels can be obtained by a valid tie-breaking rule.
\end{proof}

\subsection{Exact solution regions and fully labelled triangles}
A fully labelled elementary triangle should not be confused with the exact solution region of Definition~\ref{def:S}. The three labels of an elementary triangle are obtained at three different points of $T$. By contrast, a point of $\Sp^{\pm}_{ijk}$ is a single allocation at which agent $i$ prefers $A$, agent $j$ prefers $B$, and agent $k$ prefers $C$ simultaneously.

Consequently, the existence of a fully labelled elementary triangle does not by itself imply that the triangle intersects the exact continuous solution region. The discrete and continuous objects encode different statements:
\begin{center}
\begin{tabular}{ll}
fully labelled triangle: & three different labels at three sampled allocations,\\
solution region: & three different preferences at the same allocation.
\end{tabular}
\end{center}

Examples~\ref{ex:inc} and~\ref{ex:dis} in Section~\ref{subsec:examples} illustrate this distinction: in both of them the unique fully labelled elementary triangle does not meet the exact solution region, while several elementary triangles that do meet it are not fully labelled.

This observation does not contradict Sperner's lemma. Sperner's lemma guarantees a fully labelled triangle under the stated boundary conditions; additional arguments are required if that triangle is to be interpreted as an exact solution of the underlying continuous allocation problem. The following results supply such an argument in the limit of fine meshes.

\begin{lemma}\label{lemma:mesh} Let $(\mathcal T_n)$ be a sequence of triangulations of $T$ with $\operatorname{mesh}(\mathcal T_n)\to0$, each equipped with an ownership rule under which the three agents are represented in every elementary triangle, and labelled as in Proposition \ref{prop:spinc} (incentive model) or Proposition \ref{prop:spdis} (disincentive model). Let $\tau_n$ be a fully labelled elementary triangle of $\mathcal T_n$ and $p_n\in\tau_n$. Then, every accumulation point of $(p_n)$ belongs to $\mathcal E^{\pm}$.
\end{lemma}

\begin{proof} Let $p=(a,b,c)$ be an accumulation point of $(p_n)$ and pass to a subsequence with $p_n\to p$. In $\tau_n$ each agent owns exactly one vertex and the three labels are distinct, so $\tau_n$ defines a bijection assigning to each agent the label of its vertex. Since there are only six bijections, after passing to a further subsequence we may assume that this bijection is the same for all $n$, say $i\mapsto A$, $j\mapsto B$, $k\mapsto C$. Let $v^r_n=(a^r_n,b^r_n,c^r_n)$ be the vertex of $\tau_n$ owned by agent $r$. Its label is an alternative weakly preferred by $r$ at $v^r_n$; in the disincentive model this also holds at the vertices of $T$, where the label is one of the two tied alternatives. Hence, in the incentive model,
$$a^i_n\ge\lambda^i_{AB}b^i_n,\quad a^i_n\ge\lambda^i_{AC}c^i_n,\quad b^j_n\ge a^j_n/\lambda^j_{AB},\quad b^j_n\ge\lambda^j_{BC}c^j_n,\quad c^k_n\ge a^k_n/\lambda^k_{AC},\quad c^k_n\ge b^k_n/\lambda^k_{BC},$$

and in the disincentive model the same inequalities hold with $\le$. Since $v^r_n$ and $p_n$ lie in $\tau_n$,
$$\|v^r_n-p\|\le\|v^r_n-p_n\|+\|p_n-p\|\le\operatorname{mesh}(\mathcal T_n)+\|p_n-p\|\longrightarrow0,$$

so $v^r_n\to p$ for $r=i,j,k$. Letting $n\to\infty$ in the six inequalities yields the six weak inequalities defining $\mathcal E^{\pm}_{ijk}$ (Definition \ref{def:E}). Therefore, $p\in\mathcal E^{\pm}_{ijk}\subseteq\mathcal E^{\pm}$.
\end{proof}

\begin{proposition}\label{prop:liminf}
Let $F_n$ denote the union of the fully labelled elementary triangles of
$\mathcal T_n$; it is a nonempty (Lemma \ref{esperner}) closed set. Under the assumptions of Lemma~\ref{lemma:mesh}. 

\begin{enumerate}
\item[(i)] every point that belongs to $F_n$ for infinitely many $n$ lies in
$\mathcal E^\pm$;
\item[(ii)] if $\mathcal S^\pm\neq\emptyset$, then
$\mathcal S^\pm\subseteq\liminf_n F_n:=\bigcup_{m\ge1}\bigcap_{n\ge m}F_n$.
\end{enumerate}
Consequently, if $\mathcal S^\pm\neq\emptyset$, then
$\mathcal S^\pm\subseteq\liminf_n F_n\subseteq\mathcal
E^\pm=\overline{\mathcal S^\pm}$,
and therefore $\overline{\liminf_n F_n}=\mathcal E^\pm$.
\end{proposition}
\begin{proof}
(i) Let $p\in F_n$ for infinitely many $n$, and for each such $n$ choose a
fully
labelled elementary triangle $\tau_n\ni p$. Applying Lemma \ref{lemma:mesh} to this
subsequence of triangulations with $p_n=p$ gives $p\in\mathcal E^\pm$.
(ii) Let $p\in\mathcal S^\pm$. Since $\mathcal S^\pm$ is defined by strict inequalities and its closure lies in $\operatorname{int}(T)$ (Proposition \ref{prop:interior}), there is $\varepsilon>0$ with $B(p,\varepsilon)\cap T\subseteq\mathcal S^\pm$. Choose $N$ with $\operatorname{mesh}(\mathcal T_n)<\varepsilon$ for $n\ge N$. Then, for $n\ge N$, every elementary triangle $\tau\ni p$ of $\mathcal T_n$ satisfies $\tau\subseteq B(p,\varepsilon)\cap T\subseteq\mathcal S^\pm$, so $\tau$ is fully labelled as it is included in the solution region. Hence $p\in F_n$ for all $n\ge N$. The last statement follows by taking closures in $\mathcal S^\pm\subseteq\liminf_nF_n\subseteq\overline{\mathcal S^\pm}$.
\end{proof}

\subsection{Worked examples}\label{subsec:examples}
In both examples we use the standard triangular subdivision of $T$ with $m=10$: its vertices are the points $\big((m-i-j)/m,\,i/m,\,j/m\big)$ with $i,j\ge0$, $i+j\le m$, and the vertex with indices $(i,j)$ is owned by agent $1+\big((i+2j)\bmod 3\big)$. Along every elementary triangle the quantity $i+2j$ takes three consecutive values, so the three agents are represented in every elementary triangle. Labels are preferred alternatives of the owner, with ties at interior vertices broken alphabetically; in the disincentive model the corners are labelled by the derangement $A\mapsto B$, $B\mapsto C$, $C\mapsto A$. All quantities below were computed in exact rational arithmetic.

\begin{example}[Incentive model: positions and bonuses]\label{ex:inc}
Three employees are to be assigned to three positions $A$, $B$, $C$, and a bonus budget normalized to $1$ is split as $(a,b,c)$ among the positions. The consistent thresholds $(\lambda_{AB},\lambda_{AC},\lambda_{BC})$ are $(5/2,\,5/6,\,1/3)$ for agent~1, $(5/2,\,5,\,2)$ for agent~2 and $(2,\,2,\,1)$ for agent~3, so that, with $w_C=1$, $w^1=(5/6,1/3,1)$, $w^2=(5,2,1)$, $w^3=(2,1,1)$, and $P^1=(5/13,2/13,6/13)$, $P^2=(5/8,1/4,1/8)$, $P^3=(1/2,1/4,1/4)$. Writing an assignment as the alternatives received by agents 1, 2, 3, the assignment scores are $\Phi(ABC)=5/3$, $\Phi(ACB)=5/6$, $\Phi(BAC)=5/3$, $\Phi(BCA)=2/3$, $\Phi(CAB)=5$ and $\Phi(CBA)=4$. The unique minimizer is $BCA$ (agent 3 $\mapsto A$, agent 1 $\mapsto B$, agent 2 $\mapsto C$), so by Theorem~\ref{thm:main} the strict solution region is $\Sp^{+}_{312}$. It is the quadrilateral with vertices $(10/19,4/19,5/19)$, $(5/8,1/4,1/8)=P^2$, $(4/7,2/7,1/7)$ and $(1/2,1/4,1/4)=P^3$.

The Sperner labelling has a unique fully labelled elementary triangle, with vertices $(3/5,3/10,1/10)$ (owner agent~3, label $A$), $(1/2,2/5,1/10)$ (agent~1, label $B$) and $(1/2,3/10,1/5)$ (agent~2, label $C$). It records exactly the optimal assignment, yet it does not meet the closure of $\Sp^{+}$. Conversely, three elementary triangles meet $\Sp^{+}$ and none of them is fully labelled. See Figure~\ref{fig:exinc}.
\end{example}

\begin{figure}[ht]
\centering
\includegraphics[width=0.86\textwidth]{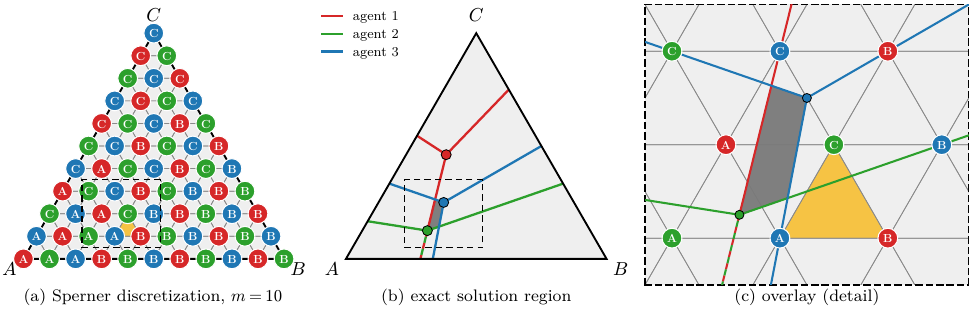}
\caption{Example~\ref{ex:inc} (incentive model). (a) Sperner-based discretization with $m=10$: the colour of each vertex indicates its owner (agents 1, 2, 3 in red, green, blue), the letter its label, and the fully labelled elementary triangle is shaded in yellow. (b) Exact solution region $\Sp^{+}$ (grey) obtained from the six thresholds. (c) Overlay of (a) and (b) in the window marked by the dashed rectangle.}
\label{fig:exinc}
\end{figure}

\begin{example}[Disincentive model: rooms and rent]\label{ex:dis}
Three housemates share a flat with rooms $A$, $B$, $C$ and a total rent normalized to $1$, split as $(a,b,c)$. The consistent thresholds $(\lambda_{AB},\lambda_{AC},\lambda_{BC})$ are $(3/2,\,15/4,\,5/2)$ for agent~1, $(1/3,\,4/9,\,4/3)$ for agent~2 and $(1,\,5/2,\,5/2)$ for agent~3, so that $w^1=(15/4,5/2,1)$, $w^2=(4/9,4/3,1)$, $w^3=(5/2,5/2,1)$, and $P^1=(15/29,10/29,4/29)$, $P^2=(4/25,12/25,9/25)$, $P^3=(5/12,5/12,1/6)$. The assignment scores are $\Phi(ABC)=5$, $\Phi(ACB)=75/8$, $\Phi(BAC)=10/9$, $\Phi(BCA)=25/4$, $\Phi(CAB)=10/9$ and $\Phi(CBA)=10/3$. The unique maximizer is $ACB$ (agent 1 $\mapsto A$, agent 3 $\mapsto B$, agent 2 $\mapsto C$), so the strict solution region is $\Sp^{-}_{132}$, the quadrilateral with vertices $(4/11,4/11,3/11)$, $(6/13,4/13,3/13)$, $(15/29,10/29,4/29)=P^1$ and $(5/12,5/12,1/6)=P^3$.

The Sperner labelling has a unique fully labelled elementary triangle, with vertices $(1/2,3/10,1/5)$ (owner agent~2, label $C$), $(1/2,1/5,3/10)$ (agent~3, label $B$) and $(2/5,3/10,3/10)$ (agent~1, label $A$). Again it records the optimal assignment but does not meet the closure of $\Sp^{-}$, while five elementary triangles meet $\Sp^{-}$ and none of them is fully labelled. See Figure~\ref{fig:exdis}.
\end{example}

\begin{figure}[ht]
\centering
\includegraphics[width=0.86\textwidth]{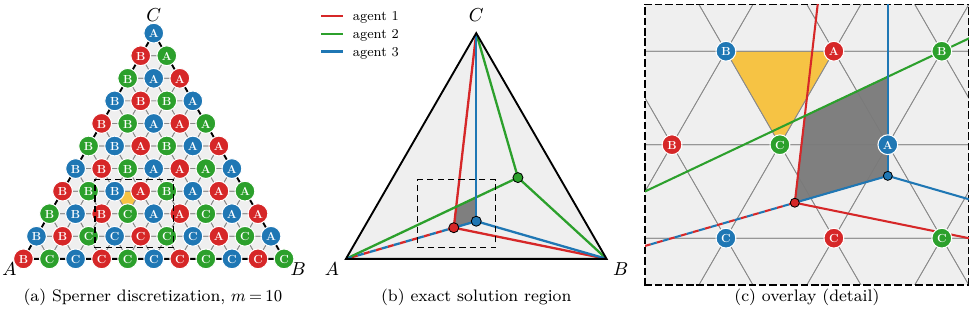}
\caption{Example~\ref{ex:dis} (disincentive model), with the same conventions as in Figure~\ref{fig:exinc}. The corners are labelled by the derangement $A\mapsto B$, $B\mapsto C$, $C\mapsto A$ of Proposition~\ref{prop:spdis}.}
\label{fig:exdis}
\end{figure}

\subsection{Local refinement and dependence on the triangulation}
A common numerical strategy is to start from a relatively coarse triangulation, locate a fully labelled elementary triangle, and then refine only that triangle. This step requires care. A triangle that is fully labelled as part of a global Sperner labelling does not automatically inherit the boundary conditions needed to regard it as a new independent Sperner domain. Its three edges are generally interior edges of the original triangulation, and the labels along them need not be restricted to the labels of their endpoints.

In Example~\ref{ex:inc} (Figure~\ref{fig:exinc}), for example, one of the edges of the selected fully labelled triangle contains points at which an agent strictly prefers the third alternative; the same happens in Example~\ref{ex:dis}. Treating that triangle as a new domain and applying Sperner's lemma locally would therefore require additional boundary assumptions that are not inherited from the original problem.

A second distinction is mesh dependence. Different initial triangulations, or different refinements of the same simplex, may produce different fully labelled elementary triangles. This is natural for a discrete approximation method. By contrast, once the reciprocal thresholds are fixed, the continuous solution region $\Sp^{\pm}$ is determined by the inequalities of Section~\ref{sec:main} and does not depend on a triangulation.

Thus a more precise comparison is not that one method ``succeeds'' while the other ``fails'', but that the two methods return different mathematical objects: Sperner-based discretization localizes candidate solutions through a mesh, whereas the reciprocal-matrix approach characterizes the exact feasible region under its parametric preference assumptions.

\subsection{Preference elicitation and number of queries}
The reciprocal-matrix formulation also changes the amount and type of information elicited from the agents. Under multiplicative consistency, each agent has only two independent thresholds, since $\lam{i}{BC}=\lam{i}{AC}/\lam{i}{AB}$. Thus, if the thresholds are elicited through pairwise comparisons, the three-agent model requires $6$ independent pairwise preference assessments.

For comparison, consider the standard triangular lattice obtained by dividing each side of $T$ into $m$ equal subintervals. The number of triangulation vertices is
\[
N_m=\binom{m+2}{2}=\frac{(m+1)(m+2)}{2}.
\]
If one preference query is made at every non-corner vertex, the corresponding discrete scheme uses at most $N_m-3$ vertex queries, because the labels of the three vertices of $T$ are fixed by the boundary convention (Propositions~\ref{prop:spinc}--\ref{prop:spdis}). This yields 63 in the example with $m=10$, which gives the comparison 6 versus 63 for the mesh of Figures~\ref{fig:exinc} and~\ref{fig:exdis}.

\subsection{Summary of the comparison}
Table~\ref{tab:comparison} summarizes the main differences between the two approaches. The comparison highlights the trade-off underlying the proposed method. The reciprocal-matrix approach gains an exact, mesh-independent description of the solution set and a compact elicitation scheme by imposing a structured parametric model of preferences. Sperner-based discretization is more flexible with respect to preference structure, but returns a mesh-dependent discrete certificate whose accuracy must be interpreted relative to the chosen triangulation.

\begin{table}[ht]
\centering\small
\begin{tabular}{>{\raggedright\arraybackslash}p{2.7cm}>{\raggedright\arraybackslash}p{5.2cm}>{\raggedright\arraybackslash}p{5.2cm}}
\toprule
 & Reciprocal-matrix method & Sperner-based discretization\\
\midrule
Preference information & Two independent pairwise thresholds per agent under consistency & Preferred alternative at sampled triangulation vertices\\
Output & Exact continuous solution region $\Sp^{\pm}$ (and the extended set $\Ex^{\pm}$, which is never empty) & An odd number of fully labelled elementary triangles, converging to $\Ex^{\pm}$ as the mesh is refined\\
Triangulation & Not required & Required\\
Dependence on mesh & None once the thresholds are fixed & Output depends on the chosen triangulation and its resolution\\
Queries in the three-agent model & 6 pairwise assessments (each a real-valued threshold) & At most $\binom{m+2}{2}-3$ vertex queries if every non-corner vertex of the standard $m$-subdivision is queried once\\
Structural assumption & Consistent reciprocal thresholds and linear indifference boundaries & Weaker local/ordinal preference information\\
\bottomrule
\end{tabular}
\caption{Comparison between the continuous reciprocal-matrix approach and a finite Sperner-based discretization. The query counts refer to the specific elicitation schemes described in the text and should not be interpreted as a universal complexity bound for all Sperner-based fair-division algorithms.}
\label{tab:comparison}
\end{table}

\section{Conclusions and future lines}\label{sec:concl}

As recalled in the Introduction, a fair-allocation problem combines two tasks: describing how agents compare the alternatives and locating an allocation at which all these comparisons are satisfied at once. Sperner-based discretizations solve the second task under very weak assumptions on the first, but what they return is a mesh-dependent certificate rather than the fair allocations themselves. In this paper we have explored the opposite trade-off: by describing the comparisons with consistent reciprocal matrices, the set of fair allocations can be computed exactly.

More precisely, we have shown that, when the preferences of three agents over three alternatives with values on the simplex are described by multiplicatively consistent reciprocal matrices, the fair-allocation problem admits an exact continuous solution. Six independent thresholds determine linear indifference boundaries; the assignment score decides which assignment, if any, admits strictly envy-free allocations, and the corresponding region is unique, convex and interior. The weak solution set is never empty and degenerates to a segment or a point exactly when the assignment score has ties. Compared with Sperner-based discretization, the method returns the exact solution set instead of a mesh-dependent certificate, at the price of a stronger parametric assumption on preferences. Fully labelled triangles converge to the extended solution set as the mesh is refined, so the two approaches agree in the limit, while at a fixed resolution a fully labelled triangle need not meet the exact solution region.

Several extensions are natural. First, inconsistent reciprocal matrices no longer derive from a single priority vector; the three indifference segments of an agent are then no longer concurrent, and gaps or overlaps between preference regions appear, which calls for a dedicated analysis in both the incentive and the disincentive models. Second, straight-line separators could be replaced by curves better adapted to specific applications, which would turn the feasibility problem into a nonlinear one. Finally, for $n>3$ agents the geometric description becomes more involved (the solution regions live in an $(n-1)$-simplex and there are $n!$ assignments). In logarithmic coordinates, however, the preferences become quasi-linear and the feasibility question reduces to a linear assignment problem, as in the rent-division literature \cite{AlkanDemangeGale,GalMashProcacciaZick}; developing this extension is left for future work.

\vspace{1cm}
\textbf{Code availability.}
The Python scripts used to generate the numerical results and figures are  available in:
\url{https://github.com/LuisAngelCalvoPascual/continuous_envy-free}.


\end{document}